\documentclass[english, 10pt]{amsart}

\usepackage{babel}
\usepackage[utf8x]{inputenc}
\usepackage{mathrsfs}
\usepackage{amsmath}
\usepackage{amssymb}
\usepackage{amsthm}
\usepackage[all]{xy}
\usepackage{leftidx}
\usepackage{amsmath,amssymb,amsfonts}
\usepackage{amsmath,amscd}
\usepackage{textcomp}

\newtheorem{defn}{Definition}
\newtheorem{thm}{Theorem}[section]
\newtheorem{prop}[thm]{Proposition}
\newtheorem{lemme}[thm]{Lemma}
\newtheorem{remark}[defn]{Remark}
\newtheorem{question}[thm]{Question}
\newtheorem{conj}[thm]{Conjecture}

\DeclareMathOperator{\Id}{Id}
\DeclareMathOperator{\Pic}{Pic}
\DeclareMathOperator{\Ricci}{Ric}
\DeclareMathOperator{\rank}{rank}
\DeclareMathOperator{\Hom}{Hom}
\DeclareMathOperator{\curv}{R}

\DeclareMathOperator{\nd}{nd}

\newcommand{\holom}[3]{\ensuremath{#1\colon #2  \rightarrow #3}}
\newcommand{\codim}{{\rm codim}}
\newcommand{\N}{\ensuremath{\mathbb{N}}}
\newcommand{\PP}{\ensuremath{\mathbb{P}}}

\newcommand\sE{{\mathcal E}}

\newcommand\sF{{\mathcal F}}
\newcommand\sG{{\mathcal G}}

\newcommand\bZ{{\mathbb Z}}

\newcommand\bQ{{\mathbb Q}}

\title[Compact K\"ahler manifolds with nef anticanonical bundles]{A remark on compact K\"ahler manifolds with nef anticanonical bundles and its applications }

\author{Junyan Cao}
\email{junyan@kias.re.kr ; jycao136@yahoo.com}
\address{KIAS\\
85 Hoegiro, Dongdaemun-gu\\
Seoul 130-722\\
Republic of Korea}

\begin{document}

\begin{abstract}
Let $(X, \omega_X)$ be a compact K\"ahler manifold such that the anticanonical bundle $− K_X$ is nef. 
We prove that the slopes of the Harder-Narasimhan filtration of the tangent bundle with respect to a polarization of the form 
$(\omega_X )^{n-1}$ are semi-positive. 
As an application, we give a characterization of rationally connected compact K\"ahler manifolds with nef anticanonical bundles. 
As another application, we give a simple proof of the surjectivity of the Albanese map.  
\end{abstract}

\maketitle

\section{Introduction}

Compact Kähler manifolds with semi-positive anticanonical bundles have been studied in depth in \cite{CDP12},
where a rather general structure theorem for this type of manifolds has been obtained.
It is a natural question to find some similar structure theorems 
for compact Kähler manifolds with nef anticanonical bundles.
Obviously, we cannot hope the same structure theorem for this type of manifolds (cf. \cite[Remark 1.7]{CDP12}).
It is conjectured that the Albanese map is a submersion 
and that the fibers exhibit no variation of their complex structure.   

In relation with the structure of compact Kähler manifolds with nef anticanonical bundles,
it is conjectured in \cite[Conj. 1.3]{Pet}
that the tangent bundles of projective manifolds with nef anticanonical bundles are generically nef.
We first recall the notion of generically semi-positive (resp. strictly positive) (cf. \cite[Section 6]{Mi87})

\begin{defn}
Let $X$ be a compact Kähler manifold and let $E$ be a vector bundle on $X$.
Let $\omega_1,\cdots,\omega_{n-1}$ be Kähler classes.
Let 
$$0\subset \mathcal{E}_{0}\subset \mathcal{E}_{1}\subset\cdots\subset \mathcal{E}_{s}=E \qquad (\text{resp. } \Omega_{X}^{1})$$
be the Harder-Narasimhan semi-stable filtration with respect to $(\omega_1,\cdots,\omega_{n-1})$.
We say that $E$ is generically $(\omega_1,\cdots,\omega_{n-1})$-semi-positive (resp. strictly positive),
if 
$$\int_{X}c_{1}(\mathcal{E}_{i+1}/\mathcal{E}_{i})\wedge\omega_1\wedge\cdots\wedge\omega_{n-1}\geq 0 
\qquad (\text{ resp.} > 0) \qquad\text{for all } i. $$
If $\omega_1=\cdots=\omega_{n-1}$, we write the polarization as $\omega_1 ^{n-1}$ for simplicity.
\end{defn}

We rephrase \cite[Conj. 1.3]{Pet} as follows

\begin{conj}
Let $X$ be a projective manifold with nef anticanonical bundle.
Then $T_X$ is generically $(H_1,\cdots,H_{n-1})$-semi-positive
for any $(n-1)$-tuple of ample divisors $H_{1},\cdots,H_{n-1}$. 
\end{conj}

In this article, we first give a partial positive answer to this conjecture.
More precisely, we prove

\begin{thm}\label{mainthm}
Let $X$ be a compact Kähler manifold with nef anticanonical bundle (resp. nef canonical bundle).
Then $T_X$ (resp. $\Omega_X ^{1}$) is generically $\omega_X ^{n-1}$-semi-positive for any Kähler class $\omega_X$.
\end{thm}

\begin{remark}
We have been informed that I. Enoki proved in \cite[Thm 1.4]{Eno} for the nef canonical bundle case by using the same method.
Although the proof for the nef anticanonical bundle case is almost the same, we still give the proof here for the convenience of readers.
\end{remark}

\begin{remark}
If $X$ is projective and $K_X$ is nef, Theorem \ref{mainthm} is a special case of \cite[Cor. 6.4]{Mi87}.
Here we prove it for arbitrary compact Kähler manifolds with nef canonical bundles. 
If $-K_X$ is nef, Theorem \ref{mainthm} is a new result even for algebraic manifolds.
\end{remark}

As an application, we give a characterization of rationally connected 
compact Kähler manifolds with nef anticanonical bundles.

\begin{prop}\label{introprop}
Let $X$ be a compact Kähler manifold with nef anticanonical bundle.
Then the following four conditions are equivalent

{\em (i):} $H^{0}(X, (T_{X}^{*})^{\otimes m})=0\qquad\text{for all }m\geq 1 .$

{\em (ii):} $X$ is rationally connected.

{\em (iii):} $T_X$ is generically $\omega_X ^{n-1}$-strictly positive for some Kähler class $\omega_X$.

{\em (iv):} $T_X$ is generically $\omega_X ^{n-1}$-strictly positive for any Kähler class $\omega_X$.
\end{prop}

\begin{remark}
We have been informed by Q. Zhang that Proposition \ref{introprop} is a special case of \cite[Cor 1]{Zha05},
The proof of \cite[Cor 1]{Zha05} uses an algebraic method, but our proof is purely analytic. 
\end{remark}

\begin{remark}
Mumford has in fact stated the following conjectured which would generalize the
first part of Proposition \ref{introprop}:
for any compact Kähler manifold $X$, $X$ is rationally connected if and only if
$$H^{0}(X, (T_{X}^{*})^{\otimes m})=0\qquad\text{for all }m\geq 1.$$
We thus prove the conjecture of Mumford under the assumption that $-K_X$ is nef. 
\end{remark}

As another application, we study the effectiveness of $c_{2}(T_{X})$. 
It is conjectured by Kawamata that

\begin{conj}
If $X$ is a compact Kähler manifold with nef anticanonical bundle.
Then  
$$\int_{X}(c_{2}(T_{X}))\wedge\omega_{1}\wedge\cdots\wedge\omega_{n-2}\geq 0,$$
for all nef classes $\omega_{1},\cdots, \omega_{n-2}$.
\end{conj}

When $\dim X=3$, this conjecture was solved by \cite{Xie05}.
Using Theorem \ref{mainthm} and an idea of A. H\"{o}ring, 
we prove

\begin{prop}
Let $(X ,\omega_X) $ be a compact Kähler manifold with nef anticanonical bundle.
Then 
\begin{equation}\label{inequintro}
\int_{X}c_{2}(T_{X})\wedge(c_1 (-K_X)+\epsilon\omega_X)^{n-2}\geq 0
\end{equation}
for every $\epsilon > 0$ small enough.
Moreover, if $X$ is projective and the equality holds for some $\epsilon > 0$ small enough,
then after a finite \'etale cover,
$X$ is either a torus or 
a smooth $\PP^1$-fibration over a torus.
\end{prop}

As the last application, 
we study the Albanese map of compact Kähler manifolds with nef anticanonical bundles.
It should be first mentioned that the surjectivity of the Albanese map 
has been studied in depth by several authors.
If $X$ is assumed to be projective, 
the surjectivity of the Albanese map was proved by Q.Zhang in \cite{Zha96}.
Still under the assumption that $X$ is projective, \cite{LTZZ10} proved that 
the Albanese map is equidimensional and all the fibres are reduced. 
Recently, M. P\u{a}un \cite{Pau12} proved
the surjectivity for arbitrary compact Kähler manifolds with nef anticanonical bundles,
as a corollary of a powerful method based on a direct image argument.
Unfortunately, it is hard to get information for the singular fibers from his proof.
Using Theorem \ref{mainthm}, we give a new proof of the surjectivity for the Kähler case,
and prove that the map is smooth outside a subvariety of codimension at least 2.

\begin{prop}
Let $X$ be a compact Kähler manifold with nef anticanonical bundle. 
Then the Albanese map is surjective, and smooth outside a subvariety of codimension 
at least~$2$. 
In particular, the fibers of the Albanese map are connected and reduced 
in codimension~$1$.
\end{prop}

\noindent{\bf Acknowledgements:} I would like to thank my supervisor J.-P. Demailly who
brought me to this interesting subject and also for his insightful criticisms and
helpful discussions during this work. Thousands of thanks to A. H\"oring for his very
helpful suggestions and comments on the drafts of this paper. I would also like to
thank Th. Peternell and Q. Zhang for their interests and comments on this work.

\section{Preparatory lemmas}

The results in this section should be well known to experts.
For the convenience of readers, we give an account of the proofs here. 
\smallskip

Since the polarization discussed in this article takes values in real coefficient, 
it is useful to discuss firstly the stable properties of the Harder-Narasimhan semi-stable filtration
with respect to certain transcendental polarization.

\begin{prop}\label{addstable}
Let $(X, \omega )$ be a compact K\"ahler manifold of dimension $n$ and 
let $\sF$ be a torsion free coherent sheaf on $X$. 
Let $\alpha\in H^{1,1}(X, \mathbb{R})$ be a big and nef class.
\footnote{We refer to \cite[Def 6.16]{Dem12} for the precise definition of big, nef class in $H^{1,1} (X, \mathbb{R})$.}
If $\sF$ is $\alpha^{n-1}$-stable, then it is stable after a small perturbation. 
\end{prop}

To prove Proposition \ref{addstable}, we need the following lemma.

\begin{lemme}\label{addlemma}
In the situation of Proposition \ref{addstable},  
we can find a basis $\{e_1, \cdots , e_s \}$ of $H^{2(n-1)}(X,\bQ)$, such that
   
{\em (i)}: $\alpha^{n-1}=\sum\limits_{i=1}^{s} \lambda_i\cdot e_i$ for some $\lambda_i > 0$.

{\em (ii)}: $e_i$ can be represented by strictly positive currents (not necessary $(n-1, n-1)$-currents), 
i.e., for every smooth positive $2$-forms $f\in C^{\infty} (X, \Omega_\mathbb{C} ^2)$,
we have $\langle e_i, f\rangle > 0$.
\end{lemme}

\begin{proof}
Thanks to \cite[Thm 2.12]{DP04}, $\alpha$ can be written as $c\omega + [T]$, where $[T]\geq 0$ in the sense of currents and $c>0$.
Combining with the fact that $\alpha $ is nef, we have
$$\alpha^{n-1} \geq (c\omega)^{n-1} \qquad\text{in the sense of currents.}$$ 
As a consequence, we know that $\alpha^{n-1} = (c\omega)^{n-1} + T_1$, where $T_1 \geq 0$ is a semi-positive $(n-1,n-1)$-current.
Since $(c\omega)^{n-1} > 0$ is a smooth $(n-1 , n-1)$-class, 
we can take a basis $\{w_i\}$ of $H^{2(n-1)}(X,\bQ)$, such that 

{\em (i)} each $w_i$ is in a neighborhood of $(c\omega)^{n-1}$.
\footnote{Therefore $w_i$ are smooth strictly positive forms, but not necessary of $(n-1,n-1)$-forms.}

{\em (ii)}  $(c\omega)^{n-1} = \sum\limits_i \lambda_i w_i$ for some $\lambda_i > 0$ and $\sum\limits_i \lambda_i =1$.

{\em (iii)} $ w_i + T_1 \in H^{2(n-1)} (X, \mathbb{Q})$ for every $i$.

It is easy to see that the basis $\{ w_i + T_1 \}$ satisfies the conditions of Lemma \ref{addlemma}.
\end{proof}

\begin{proof}[proof of Proposition \ref{addstable}]

For every coherent sheaf $\sG$, set 
$$\mu (\sG):=\frac{1}{\rank (\sG)}\int_{X}c_{1}(\sG)\wedge\alpha^{n-1}$$
be the slope of $\sG$ with respect to $\alpha^{n-1}$.
Let $\{ e_i\}$ be the basis constructed in Lemma \ref{addlemma}.
To prove the proposition,
it is sufficient to prove that
$$\sup \{\mu (\sG) | 
 \sG \text{ a coherent subsheaf of } \sF \text{ with strictly smaller rank }\}$$
is strictly smaller than $\mu (\sF)$.

If it is not true, then there exists a sequence of subsheaves $\{\sG_i\}_{i=1}^{+\infty}$ 
of $\sF$, such that $\lim\limits_{i\rightarrow +\infty} \mu (\sG_i) =\mu (\sF)$.
Set $a_{i,j} : = c_1 (\sG_i) \cdot e_j$. 
By the proof of \cite[Lemma 7.16, Chapter 5]{Kob87}, 
there exists a uniform constant $M$ such that
\begin{equation}\label{upperbound}
 a_{i,j } \leq M \qquad\text{for every } i,j.
\end{equation}
By $(i)$ of Lemma \ref{addlemma}, we have 
$$\rank (\sG_i) \cdot \mu (\sG_i )= c_1 (\sG_i) \cdot (\sum_j \lambda_j \cdot e_j) = \sum_j a_{i,j} \cdot \lambda_j .$$
Since $\mu (\sG_i ) $ is uniformly bounded from below, 
$ \sum\limits_j \lambda_j a_{i,j} $ is also uniformly bounded from below.
Combining with the fact that $\lambda_j >0$ and \eqref{upperbound}, 
we know that $\{ a_{i,j }\}_{i, j}$ is also uniformly bounded from below, 
i.e., there exists a uniform constant $N$ such that 
\begin{equation}\label{lowbound}
 a_{i,j } \geq N \qquad\text{for every } i,j.
\end{equation}
Noting that $c_1 (\sG_i) \in H^2 (X, \mathbb{Z})$,  Then \eqref{upperbound} and \eqref{lowbound} implies that 
$\{ a_{i,j }\}_{i, j}$ contains only finite different elements.
Therefore $\{\mu (\sG_i )\}_{i=1}^{+\infty}$ contains only finite different values.
We have thus the contradiction.
\end{proof}

We are also interested in the following situation.

\begin{prop}\label{lemmastable}
Let $(X, \omega )$ be a compact K\"ahler manifold of dimension $n$ and let $\sF$ 
be a torsion free coherent sheaf on $X$. 
Let $D_1,\cdots ,D_{n-1}$ be nef classes in $H^{1,1}(X, \bQ)$ and let $A$ be a K\"{a}hler class.
Let $a$ be a sufficiently small positive number.
Then the Harder-Narasimhan semi-stable filtration of $E$ with respect to 
$( D_1+a\cdot A, \cdots , D_{n-1}+ a\cdot A )$
is independent of $a$.
\end{prop}

\begin{remark}
If $A$ has rational coefficients, Proposition \ref{lemmastable} is proved in \cite{KMM04}.
When $A$ is not necessarily rational, the proof turns out to be more complicated.
We begin with the following easy observation.
\end{remark}

\begin{lemme}\label{auxi}
In the situation of Proposition \ref{lemmastable}, fix a $k\in \{0, 1 ,2,\cdots, n\}$. 
Then we can find a basis $\{e_1, \cdots , e_s \}$ of $H^{2k}(X,\bQ)$ depending only on $A^{k}$, such that
   
{\em (i)}: $A^{k}=\sum\limits_{i=1}^{s} \lambda_i\cdot e_i$ for some $\lambda_i > 0$.

{\em (ii)}: Let $\sF$ be a torsion free coherent sheaf.
Set 
$$D^{t}:=\sum_{i_1< i_2 < \cdots < i_{t}} D_{i_1}\cdot D_{i_2}\cdots D_{i_{t}}\qquad\text{for any }t,$$
and
\begin{equation}\label{orthprojection}
a_i (\sF):= c_1 (\mathcal{F})\cdot D^{n-k-1}\cdot e_i .
\end{equation}
Then the subset $S$ of the set $\bQ$ of rational numbers such that
$$S := \{a_i (\mathcal{F}) |\text{  }\mathcal{F}\subset E, i\in \{1,\cdots s\} \}$$
is bounded from above, 
and the denominator (assumed positive) of all elements of $S$ is uniformly bounded from above.  
Moreover, if $\{ \sF_t \}_t$ is a sequence of coherent subsheaves of $E$ such that 
the set $\{c_1 (\sF_t)\cdot D^{n-k-1}\cdot A^k \}_t $ is bounded from below, 
then $\{c_1 (\sF_t)\cdot D^{n-k-1}\cdot A^k \}_t $ is a finite subset of $\bQ$. 
\end{lemme}

\begin{proof}
We can take a basis $\{e_i\}_{i=1}^{s}$ of $H^{2k}(X,\bQ)$ in a neighborhood of $A^{k}$, 
such that 
$$A^{k}=\sum\limits_{i=1}^{s} \lambda_i\cdot e_i\qquad \text{for some } \lambda_i > 0$$
and $(e_i)^{k,k}$ can be represented by a smooth $(k,k)$-positive form on $X$ (cf. \cite[Chapter 3, Def 1.1]{Dem}
for the definition of $(k,k)$-positivity),
where $(e)^{k,k}$ is the projection of $e$ in $H^{k,k}(X, \mathbb{R})$.
We now check that $\{e_i\}_{i=1}^{s}$ satisfies the lemma.
By construction, $(i)$ is satisfied. 
As for $(ii)$, 
since $e_i$ and $D_i$ are fixed and $c_1 (\sF)\in H^{1,1}(X, \bZ)$,
the denominator of any elements in $S$ is uniformly bounded from above.
Thanks to \eqref{orthprojection}, 
we know that $S$ is bounded from above by using the same argument as in \cite[Lemma 7.16, Chapter 5]{Kob87}.
For the last part of $(ii)$, since
$$c_1 (\sF_t)\cdot D^{n-k-1}\cdot A^k=\sum_i a_i(\sF_t)\cdot \lambda_i ,$$
we obtain that $\{ \sum\limits_i a_i(\sF_t)\cdot \lambda_i \}_{t}$ is uniformly bounded.
Since $\lambda_i > 0$ and $a_i(\sF_t)$ is uniformly upper bounded, we obtain that
$a_i (\sF_t)$ is uniformly bounded.
Combining this with the fact already proved that the denominator of any elements in $S$ is uniformly bounded,
$\{c_1 (\sF_t)\cdot D^{n-k-1}\cdot A^k \}_t $ is thus finite.
The lemma is proved.
\end{proof}

\begin{proof}[Proof of Proposition \ref{lemmastable}]
Let $\{ a_p\}_{p=1}^{+\infty}$ be a decreasing positive sequence converging to $0$.
Let $\sF_p\subset E$ be the first piece of 
the Harder-Narasimhan filtration of $E$ with respect to $(D_1+a_p \cdot A , \cdots , D_{n-1}+ a_p \cdot A)$.
Set $D^k := \sum\limits_{i_1< i_2< \cdots< i_{k}} D_{i_1}\cdot D_{i_2}\cdots D_{i_{k}}$.
Then
$$c_1 (\sF_p)\wedge(D_1+a_p\cdot A)\wedge \cdots \wedge (D_{n-1}+ a_p\cdot A)
=\sum_{k=0}^{n} ( a_p )^{k}\cdot c_1(\sF_p)\wedge D^{n-k-1}\wedge A^{k} .$$
By passing to a subsequence, we can suppose that $\rank \sF_p$ is constant.
To prove Proposition \ref{lemmastable}, 
it is sufficient to prove that for any $k$, 
after passing to a subsequence,
the intersection number 
$\{c_1(\sF_p)\wedge D^{n-k-1}\wedge A^{k}\}_p$ 
is stationary when $p$ is large enough
\footnote{In fact, if $\sF\subset E$ is always the first piece of semi-stable filtration with respect to the polarization
$(D_1+a_p \cdot A , \cdots , D_{n-1}+ a_p \cdot A)$ for a positive sequence $\{ a_p \}_{p=0}^{+\infty}$ converging to $0$, 
and $\sG\subset E$ is always the first piece of semi-stable filtration for another sequence $\{b_p \}_{p=0}^{+\infty}$ converging to $0$,
the stability condition implies that
\begin{equation}
\rank (\sG)\cdot c_1(\sF)\cdot D^{k}\cdot A^{n-k-1}=\rank (\sF)\cdot c_1(\sG)\cdot D^{k}\cdot A^{n-k-1}
\end{equation}
for any $k$.
Therefore $\sG$ has the same slope as $\sF$ with respect to $(D_1+a \cdot A , \cdots , D_{n-1}+ a \cdot A)$
for any $a> 0$.
Then $\sF=\sG$.}.

We prove it by induction on $k$. 
Note first that, by \cite[Lemma 7.16, Chapter 5]{Kob87}, 
the set $\{c_1(\sF_p)\wedge D^{n-k-1}\wedge A^{k}\}_{p, k}$ is upper bounded.
If $k=0$,
since 
$$c_1 (\sF_p)\wedge(D_1+a_p\cdot A)\wedge \cdots \wedge (D_{n-1}+ a_p\cdot A)$$
$$
\geq \frac{\rank (\sF_p)}{\rank E}c_1 (E)\wedge(D_1+a_p\cdot A)\wedge \cdots \wedge (D_{n-1}+ a_p\cdot A) ,$$
and $\lim\limits_{p\rightarrow +\infty} a_p=0$, 
the upper boundedness of $\{c_1(\sF_p)\wedge D^{n-k-1}\wedge A^{k}\}_{p, k}$ implies that
the set $\{c_1 (\sF_p)\wedge D^{n-1}\}_{p=1}^{\infty}$ is bounded from below.
Then $(ii)$ of Lemma \ref{auxi} implies that $\{c_1 (\sF_p)\wedge D^{n-1}\}_{p=1}^{\infty}$ is a finite set.
By the pigeon hole principle, after passing to a subsequence,
the set $\{c_1 (\sF_p)\wedge D^{n-1}\}_{p=1}^{\infty}$
is stationary.
Now we suppose that $\{c_1 (\sF_p)\wedge D^{n-t-1}\wedge A^{t}\}_{p}$ is constant for $p\geq p_0$, where $t\in \{0,\cdots, k-1\}$.
Our aim is to prove that after passing to a subsequence, 
$$\{c_1 (\sF_p)\wedge D^{n-k-1}\wedge A^{k}\}_{p=1}^{\infty}$$ 
is stationary.
By definition, we have
$$c_1 (\sF_p)\wedge(D_1+a_p\cdot A)\wedge \cdots \wedge (D_{n-1}+ a_p\cdot A)$$
$$
\geq c_1 (\sF_{p_0})\wedge(D_1+a_p\cdot A)\wedge \cdots \wedge (D_{n-1}+ a_p\cdot A) \qquad \text{for any }p\geq p_0.$$
Since $\{c_1 (\sF_p)\wedge D^{n-t-1}\wedge A^{t}\}_{p}$ is constant for $p\geq p_0$ when $t\in \{0,\cdots, k-1\}$,
we obtain
\begin{equation}
c_1(\sF_p)\wedge D^{n-k-1}\wedge A^{k} +\sum_{i\geq 1} (a_p) ^i\cdot c_1(\sF_p)\wedge D^{n-k-1-i}\wedge A^{k+i}\geq 
\end{equation}
$$c_1(\sF_{p_0})\wedge D^{n-k-1}\wedge A^{k} +\sum_{i\geq 1} (a_p) ^i\cdot c_1(\sF_{p_0})\wedge D^{n-k-1-i}\wedge A^{k+i}$$
for any $p\geq p_0$.
Therefore
the upper boundedness of $\{c_1(\sF_p)\wedge D^{n-k-1-i}\wedge A^{k+i}\}_{p, i}$ implies that 
$\{c_1 (\sF^p)\wedge D^{n-k-1}\wedge A^{k}\}_{p=1}^{+\infty}$ is lower bounded.
Therefore 
$$\{c_1 (\sF^p)\wedge D^{n-k-1}\wedge A^{k}\}_{p=1}^{+\infty}$$ 
is uniformly bounded.
Using $(ii)$ of Lemma \ref{auxi}, $\{c_1 (\sF^p)\wedge D^{n-k-1}\wedge A^k\}_{p=1}^{+\infty}$ is a finite set.
By the pigeon hole principle, 
after passing to a subsequence,
$$\{c_1 (\sF_p) \wedge D^{n-k-1}\cdot A^{k}\}_{p=1}^{\infty}$$ 
is stationary.
The lemma is proved.
\end{proof}

We recall a regularization lemma proved in \cite[Prop. 3]{Jac10}.

\begin{lemme}\label{lemmaregular}
 Let $E$ be a vector bundle on a compact complex manifold $X$ and 
$\mathcal{F}$ be a subsheaf of $E$ with torsion free quotient.
Then after a finite number of blowups $\pi:\widetilde{X}\rightarrow X$,
there exists a holomorphic subbundle $F$ of $\pi^{*}(E)$ containing $\pi^{*}(\mathcal{F})$
with a holomorphic quotient bundle,
such that $\pi_{*}(F)=\mathcal{F}$ in codimension $1$.
\end{lemme}

We need another lemma which is proved in full generality in \cite[Prop. 1.15]{DPS94}. 
For completeness, we give the proof here in an over simplified case, 
but the idea is the same. 
 
\begin{lemme}\label{lemmaestimatexten}
Let $(X,\omega)$ be a compact Kähler manifold.
Let $E$ be an extension of two vector bundles $E_{1}, E_{2}$
$$0\rightarrow E_{1}\rightarrow E \rightarrow E_{2}\rightarrow 0.$$
We suppose that there exist two smooth metrics $h_{1}, h_{2}$ on $E_1$ and $E_2$,
such that
\begin{equation}\label{equationcurvaturesmooth}
\frac{i\Theta_{h_{1}}(E_{1})\wedge\omega^{n-1}}{\omega^{n}}\geq c_{1}\cdot\Id_{E_{1}}\qquad
\text{and}\qquad
\frac{i\Theta_{h_{2}}(E_{2})\wedge\omega^{n-1}}{\omega^{n}}\geq c_{2}\cdot\Id_{E_{2}}
\end{equation}
pointwise.
Then for any $\epsilon > 0$, 
there exists a smooth metric $h_{\epsilon}$ on $E$ such that
$$\frac{i\Theta_{h_{\epsilon}}(E)\wedge\omega^{n-1}}{\omega^{n}}
\geq (\min (c_{1}, c_{2})-\epsilon)\cdot\Id_{E} ,$$
and 
\begin{equation}\label{add4}
\|i\Theta_{h_{\epsilon}}(E) \|_{L^{\infty}}\leq C\cdot (\|i\Theta_{h_{1}}(E_{1})\|_{L^{\infty}} + \|i\Theta_{h_{2}}(E_{2})\|_{L^{\infty}})
\end{equation}
for some uniform constant $C$ independent of $\epsilon$.
\end{lemme}

\begin{proof}
Let $[E]\in H^{1}(X, \Hom (E_{2}, E_{1}))$ be the element representing $E$ 
in the extension group.
Let $E_{s}$ be another extension of $E_{1}$ and $E_{2}$,
such that $[E_{s}]=s\cdot [E]$, where $s\in \mathbb{C}^{*}$.
Then there exists an isomorphism between these two vector bundles
(cf. \cite[Remark 14.10, Chapter V]{Dem}).
We denote the isomorphism by
$$\varphi_{s}: E\rightarrow E_{s} .$$
Thanks to \eqref{equationcurvaturesmooth}, 
if $| s |$ is small enough with respect to $\epsilon$, 
we can find a smooth metric $h_{s}$ on $E_{s}$ satisfying 
\begin{equation}\label{equation55}
\frac{i\Theta_{h_{s}}(E_{s})\wedge\omega^{n-1}}{\omega^{n}}\geq 
(\min (c_{1}, c_{2})-\epsilon)\cdot\Id_{E_{s}}
\end{equation}
and 
\begin{equation}\label{addvv1}
\|i\Theta_{h_{s}}(E_s) \|_{L^{\infty}}\leq C\cdot (\|i\Theta_{h_{1}}(E_{1})\|_{L^{\infty}} + \|i\Theta_{h_{2}}(E_{2})\|_{L^{\infty}})
\end{equation}
for some uniform constant $C$
(cf. \cite[Prop 14.9, Chapter V]{Dem}).
Let $h=\varphi_{s}^{*} (h_{s})$ be the induced metric on $E$. 
Then for any $\alpha\in E$,
\begin{equation}\label{addvv2}
\frac{\langle i\Theta_{h}(E)\alpha, \alpha \rangle_{h}}{\langle\alpha, \alpha\rangle_{h}}
=\frac{\langle \varphi_{s}^{-1}\circ i\Theta_{h_{s}}(E_{s})\varphi_{s}(\alpha), \alpha \rangle_{h}}
{\langle\alpha, \alpha\rangle_{h}}
\end{equation}
$$=\frac{\langle i\Theta_{h_{s}}(E_{s})\varphi_{s}(\alpha), \varphi_{s}(\alpha) \rangle_{h_{s}}}
{\langle\varphi_{s}(\alpha), \varphi_{s}(\alpha)\rangle_{h_{s}}} .$$ 
Combining this with \eqref{equation55}, we get
$$\frac{\langle i\Theta_{h}(E)\alpha, \alpha \rangle_{h} \wedge\omega^{n-1}}
{\langle\alpha, \alpha\rangle_{h}\cdot\omega^{n}}
\geq (\min (c_{1}, c_{2})-\epsilon)\cdot\Id_{E}.$$
Moreover, \eqref{addvv2} implies also \eqref{add4}.
The lemma is proved.
\end{proof}

We recall the following well-known equality in K\"ahler geometry.

\begin{prop}\label{ricciequality}
Let $(X,\omega_X)$ be a K\"ahler manifold of dimension $n$, $\curv$ be the curvature tensor and $\Ricci$ be the Ricci tensor
(cf. the definition of \cite[Section 7.5]{Zhe}).
Let $i \Theta_{\omega_X} (T_X)$ be the curvature of $T_X$ induced by $\omega_X$.
We have
$$\langle \frac{i\Theta_{\omega_{X}}(T_{X})\wedge \omega_{X}^{n-1}}{\omega_{X}^{n}}u, v\rangle_{\omega_X}
=\Ricci (u, \overline{v}) .$$
\end{prop}

\begin{proof}
Let $\{ e_i \}_{i=1}^{n}$ be an orthonormal basis of $T_X$ with respect to $\omega_{X}$.
By definition, we have
$$\langle \frac{i\Theta_{\omega_{X}}(T_{X})\wedge \omega_{X}^{n-1}}{\omega_{X}^{n}} u, v\rangle_{\omega_X}
=\sum\limits_{1 \leq i\leq n} \langle i\Theta_{\omega_X} (T_X) u, v\rangle (e_i, \overline{e}_i)=
\sum\limits_{1 \leq i\leq n} \curv (e_i, \overline{e}_i, u, \overline{v}).$$
By definition of the Ricci curvature (cf. \cite[Page 180]{Zhe}), we have
$$\Ricci (u, \overline{v}) =\sum\limits_{1 \leq i\leq n} \curv ( u, \overline{v}, e_i, \overline{e}_i).$$
Combining this with the First Bianchi equality
$$\sum\limits_{1 \leq i\leq n} \curv (e_i, \overline{e}_i, u, \overline{v})
=\sum\limits_{1 \leq i\leq n} \curv ( u, \overline{v}, e_i, \overline{e}_i) ,$$
the proposition is proved.
\end{proof}

\section{Main theorem }

We first prove Theorem \ref{mainthm} in the case when $-K_X$ is nef.

\begin{thm}\label{mainthmantinef}
Let $(X, \omega)$ be a compact $n$-dimensional 
Kähler manifold with nef anticanonical bundle.
Let 
\begin{equation}\label{mainfiltraion}
 0= \mathcal{E}_{0}\subset \mathcal{E}_{1}\subset\cdots\subset \mathcal{E}_{s}=T_{X}
\end{equation}
be a filtration of torsion-free subsheaves such that $\mathcal{E}_{i+1}/\mathcal{E}_{i}$
is an $\omega$-stable torsion-free subsheaf of $T_{X}/\mathcal{E}_{i}$ of maximal slope
\footnote{Using Proposition \ref{addstable}, 
one can prove the existence of such a filtration by a standard argument \cite{HN}.}.
Let 
$$\mu(\mathcal{E}_{i+1}/\mathcal{E}_{i})=\frac{1}{\rank (\mathcal{E}_{i+1}/\mathcal{E}_{i} )}
\int_{X}c_{1}(\mathcal{E}_{i+1}/\mathcal{E}_{i})\wedge\omega^{n-1}$$
be the slope of $\mathcal{E}_{i+1}/\mathcal{E}_{i}$ with respect to $\omega^{n-1}$.
Then 
$$\mu(\mathcal{E}_{i+1}/\mathcal{E}_{i})\geq 0 \qquad\text{for all } i.$$
\end{thm}

\begin{proof}

We first consider a simplified case. 

{\em Case 1 : \eqref{mainfiltraion} is regular, i.e., 
all $\mathcal{E}_{i}$, $\mathcal{E}_{i+1}/\mathcal{E}_{i}$ are vector bundles.}   

By the stability condition, to prove the theorem, it is sufficient to prove that 
\begin{equation}\label{auxilipositive}
\int_{X}c_{1}(T_X / \mathcal{E}_{i})\wedge\omega^{n-1}\geq 0 \qquad\text{for any } i.
\end{equation}
Thanks to the nefness of $-K_{X}$, 
for any $\epsilon > 0$, 
there exists a Kähler metric $\omega_{\epsilon}$ in the same class of $\omega$
such that (cf. the proof of \cite[Thm. 1.1]{DPS93})
\begin{equation}\label{equation56}
\Ricci_{\omega_{\epsilon}}\geq -\epsilon \omega_{\epsilon},
\end{equation}
where $\Ricci_{\omega_{\epsilon}}$ is the Ricci curvature with respect to the metric $\omega_{\epsilon}$.
Thanks to Proposition \ref{ricciequality},
we have
$$\langle \frac{i\Theta_{\omega_{\epsilon}}(T_{X})\wedge \omega_{\epsilon}^{n-1}}{\omega_{\epsilon}^{n}}\alpha, \alpha\rangle_{\omega_{\epsilon}}
=\Ricci_{\omega_\epsilon} (\alpha, \overline{\alpha}).$$
Then \eqref{equation56} implies a pointwise estimate
\begin{equation}\label{equation57}
\frac{i\Theta_{\omega_{\epsilon}}(T_{X})\wedge \omega_{\epsilon}^{n-1}}{\omega_{\epsilon}^{n}}
\geq -\epsilon\cdot \Id_{T_{X}} .
\end{equation}
Taking the induced metric on $T_{X}/ \mathcal{E}_{i}$ 
(we also denote it by $\omega_{\epsilon}$),
we get (cf. \cite[Chapter V]{Dem})
\begin{equation}\label{equation58}
\frac{i\Theta_{\omega_{\epsilon}}(T_{X}/\mathcal{E}_{i})
\wedge \omega_{\epsilon}^{n-1}}{\omega_{\epsilon}^{n}}
\geq -\epsilon\cdot\Id_{T_{X}/\mathcal{E}_{i}}. 
\end{equation}
Therefore
$$\int_{X}c_{1}(T_{X}/\mathcal{E}_{i})\wedge \omega_{\epsilon}^{n-1}\geq -\rank (T_{X}/\mathcal{E}_{i})\cdot \epsilon\int_X \omega_{\epsilon}^n.$$
Combining this with the fact that $[\omega_{\epsilon}]=[\omega]$, 
we get
\begin{equation}\label{equation59}
\int_{X}c_{1}(T_{X}/ \mathcal{E}_{i})\wedge\omega^{n-1}
=\int_{X}c_{1}(T_{X}/ \mathcal{E}_{i})\wedge\omega_{\epsilon}^{n-1}\geq -C\epsilon,
\end{equation}
for some constant $C$.
Letting $\epsilon\rightarrow 0$, \eqref{auxilipositive} is proved.

{\em Case 2: The general case} 

By Lemma \ref{lemmaregular}, 
there exists a desingularization $\pi: \widetilde{X}\rightarrow X$, such that 
$\pi^{*}(T_{X})$ admits a filtration:
\begin{equation}\label{regularfiltr}
0\subset E_{1}\subset E_{2}\subset\cdots \subset \pi^{*}(T_{X}) ,
\end{equation}
where $E_{i}, E_{i}/E_{i-1}$ are vector bundles 
and $\pi_{*}(E_{i})=\mathcal{E}_{i}$ outside 
an analytic subset of codimension at least 2.
Let $\widetilde{\mu}$ be the slope with respect to $\pi^{*}(\omega)$.
Then 
\begin{equation}\label{equation60}
\widetilde{\mu} (E_{i}/E_{i-1}) =\mu (\mathcal{E}_{i}/\mathcal{E}_{i-1}) 
\end{equation}
(cf. \cite[Lemma 2]{Jac10} ),
and $E_{i}/E_{i-1}$ is a $\pi^{*}(\omega)$-stable subsheaf of $\pi^{*}(T_{X})/E_{i-1}$ of maximal slope

We now prove that 
$\widetilde{\mu} (E_{i}/E_{i-1})\geq 0$.
Thanks to \eqref{equation57}, for any $\epsilon> 0$ small enough, 
we have 
$$\frac{i\Theta_{\pi^{*}\omega_{\epsilon}}(\pi^{*}(T_{X}))
\wedge (\pi^{*}\omega_{\epsilon})^{n-1}}{(\pi^{*}\omega_{\epsilon})^{n}}
\geq -\epsilon\cdot \Id_{\pi^{*}(T_{X})} ,$$
which implies that
\begin{equation}\label{equationestimationcurv}
\frac{i\Theta_{\pi^{*}\omega_{\epsilon}}(\pi^{*}(T_{X})/E_{i})\wedge 
(\pi^{*}\omega_{\epsilon})^{n-1}}{(\pi^{*}\omega_{\epsilon})^{n}}
\geq -\epsilon\cdot \Id_{\pi^{*}(T_{X})/E_{i}} .
\end{equation}
By the same argument as in Case 1, 
\eqref{equationestimationcurv} and the maximal slope condition of 
$E_{i+1}/E_{i}$ in $\pi^{*}(T_{X})/E_{i}$
implies that
$$\widetilde{\mu} (E_{i+1}/E_{i})=\frac{1}{\rank (E_{i+1}/E_{i} )}
\int_{\widetilde{X}}c_{1}(E_{i+1}/E_{i})\wedge \pi^{*}\omega^{n-1}\geq -C\epsilon$$
for some constant $C$ independent of $\epsilon$.
Letting $\epsilon\rightarrow 0$, 
we get $\widetilde{\mu} (E_{i+1}/E_{i})\geq 0$.
Combining this with \eqref{equation60}, the theorem is proved.
\end{proof}

We now prove Theorem \ref{mainthm} in the case when $K_X$ is nef.

\begin{thm}\label{canonicalcase}
Let $(X, \omega)$ be a compact Kähler manifold with nef canonical bundle.
Let 
\begin{equation}\label{mainfiltrationdualcase}
0\subset \mathcal{E}_{0}\subset \mathcal{E}_{1}\subset\cdots\subset \mathcal{E}_{s}=\Omega_{X}^{1}
\end{equation}
be a filtration of torsion-free subsheaves such that $\mathcal{E}_{i+1}/\mathcal{E}_{i}$
is an $\omega$-stable torsion-free subsheaf of $T_{X}/\mathcal{E}_{i}$ of maximal slope.
Then 
$$\int_{X}c_{1}(\mathcal{E}_{i+1}/\mathcal{E}_{i})\wedge\omega^{n-1}\geq 0 \qquad\text{for all } i. $$
\end{thm}

\begin{proof}

The proof is almost the same as Theorem \ref{mainthmantinef}.
First of all, since $K_X$ is nef, for any $\epsilon > 0$,
there exists a smooth function $\psi_{\epsilon}$ on $X$, such that
$$\Ricci_{\omega}+ i \partial\overline{\partial}\psi_{\epsilon}\leq \epsilon\omega .$$
By solving the Monge-Ampère equation
\begin{equation}\label{equation62}
(\omega+i \partial\overline{\partial}\varphi_{\epsilon})^{n}
=\omega^{n}\cdot e^{-\psi_{\epsilon}-\epsilon\varphi_{\epsilon}}, 
\end{equation}
we can construct a new K\"{a}hler metric $\omega_{\epsilon}$ 
in the cohomology class of $\omega$:
$$\omega_{\epsilon} := \omega+ i \partial\overline{\partial}\varphi_{\epsilon} .$$
Thanks to \eqref{equation62}, we have
$$\Ricci_{\omega_{\epsilon}}=\Ricci_{\omega}+i \partial\overline{\partial}\psi_{\epsilon}
+\epsilon i \partial\overline{\partial}\varphi_{\epsilon}$$
$$\leq \epsilon\omega +\epsilon i \partial\overline{\partial}\varphi_{\epsilon}=\epsilon\omega_{\epsilon} .$$

We first suppose that \eqref{mainfiltrationdualcase} is regular, i.e., 
$\mathcal{E}_{i}$ and $\mathcal{E}_{i+1}/\mathcal{E}_{i}$ are free for all $i$.
Let $\alpha\in \Omega_{X, x}^{1}$ for some point $x\in X$ with norm $\| \alpha \|_{\omega_{\epsilon}}=1$ and
let $\alpha^{*}$ be the dual of $\alpha$ with respect to $\omega_{\epsilon}$.
Then we have also a pointwise estimate at $x$:
$$\langle \frac{i\Theta_{\omega_{\epsilon}}(\Omega_{X})\wedge \omega_{\epsilon}^{n-1}}{\omega_{\epsilon}^{n}}\alpha, \alpha\rangle
=\langle -\frac{i\Theta_{\omega_{\epsilon}}(T_{X})\wedge \omega_{\epsilon}^{n-1}}{\omega_{\epsilon}^{n}}\alpha^{*}, \alpha^{*}\rangle$$
$$=-\Ricci_{\omega_{\epsilon}}(\alpha^{*},\alpha^{*})\geq -\epsilon .$$
By the same proof as in Theorem \ref{mainthmantinef}, 
$\int_{X}c_{1}(\mathcal{E}_{i+1}/\mathcal{E}_{i})\wedge\omega^{n-1}$ is semi-positive for any $i$.
For the general case, the proof follows exactly the same line as in Theorem \ref{mainthmantinef}.
\end{proof}

\section{Applications}

As an application, we give a characterization of rationally connected 
compact Kähler manifolds with nef anticanonical bundles.

\begin{prop}\label{propequivalent}
Let $X$ be a compact Kähler manifold with nef anticanonical bundle.
Then the following four conditions are equivalent

{\em (i):} $H^{0}(X, (T_{X}^{*})^{\otimes m})=0\qquad\text{for all }m\geq 1 .$

{\em (ii):} $X$ is rationally connected.

{\em (iii):} $T_X$ is generically $\omega_X ^{n-1}$ strictly positive for some Kähler class $\omega_X$.

{\em (iv):} $T_X$ is generically $\omega_X ^{n-1}$ strictly positive for any Kähler class $\omega_X$.
\end{prop}

\begin{proof}
The implications $(iv)\Rightarrow (iii)$, $(ii)\Rightarrow (i)$ are obvious.
For the implication $(iii)\Rightarrow (ii)$, we first note that $(iii)$ implies $(i)$ by Bochner technique.
Therefore $X$ is projective and any Kähler class can be approximated by rational Kähler classes. 
Using \cite[Theorem 0.1]{BM01}, $(iii)$ implies 
$(ii)$.

We now prove that $(i)\Rightarrow (iv)$. Let $\omega$ be any Kähler class.
Let 
\begin{equation}\label{filtrationmumford}
0\subset \mathcal{E}_{0}\subset \mathcal{E}_{1}\subset\cdots\subset \mathcal{E}_{s}=T_{X}
\end{equation}
be the Harder-Narasimhan semi-stable filtration with respect to $\omega^{n-1}$.
To prove $(iv)$, it is sufficient to prove 
$$\int_{X}c_{1}(T_{X}/\mathcal{E}_{s-1})\wedge\omega^{n-1}> 0 .$$
Recall that Theorem \ref{mainthmantinef} implies already that
$$\int_{X}c_{1}(T_{X}/\mathcal{E}_{s-1})\wedge\omega^{n-1}\geq 0 .$$
We suppose by contradiction that 
\begin{equation}\label{equality}
\int_{X}c_{1}(T_{X}/\mathcal{E}_{s-1})\wedge\omega^{n-1}=0 .
\end{equation}
Let $\alpha\in H^{1,1}(X, \mathbb{R})$. 
We define new Kähler metrics $\omega_{\epsilon}=\omega+\epsilon\alpha$
for $|\epsilon |$ small enough.
Thanks to \cite[Cor. 2.3]{Mi87}, the $\omega_{\epsilon}^{n-1}$-semi-stable filtration of $T_X$ is
a refinement of \eqref{filtrationmumford}.
Therefore, Theorem \ref{mainthmantinef} implies that
$$\int_{X}c_{1}(T_{X}/\mathcal{E}_{s-1})\wedge (\omega+\epsilon\alpha)^{n-1}\geq 0$$
for $|\epsilon|$ small enough.
Then \eqref{equality} implies that
$$\int_{X}c_{1}(T_{X}/\mathcal{E}_{s-1})\wedge\omega^{n-2}\wedge\alpha=0\qquad\text{for any }\alpha\in H^{1,1}(X, \mathbb{R}).$$
By the Hodge index theorem, we obtain that $c_{1}(T_{X}/\mathcal{E}_{s-1})=0$.
By duality, 
there exists a subsheaf $\sF\subset \Omega_X ^1$, 
such that 
\begin{equation}\label{equation65}
 c_{1}(\sF)=0 \qquad \text{and} \qquad \det\sF\subset (T_{X}^{*})^{\otimes\rank \sF} .
\end{equation}
Observing that $H^{1}(X, \mathcal{O}_X)=0$ by assumption, i.e., 
the group $\Pic^{0}(X)$ is trivial,
hence
\eqref{equation65} implies the existence of an integer $m$ such that
$(\det\sF)^{\otimes m}$ is a trivial line bundle.
Observing moreover that
$(\det\sF)^{\otimes m}\subset (T_{X}^{*})^{\otimes m\cdot\rank \sF}$,
then  
$$ H^{0}(X, (T_{X}^{*})^{\otimes m\cdot\rank \sF})\neq 0 ,$$
which contradicts with $(i)$.
The implication $(i)\Rightarrow (iv)$ is proved.
\end{proof}

\begin{remark}
One can also prove the implication $(iii)\Rightarrow (ii)$ without using the profound theorem of \cite{BM01}.
We give the proof in Appendix \ref{bochner}.
\end{remark}

The above results lead to the following question about rationally connected manifolds with nef anticanonical bundles.

\begin{question}\label{questionrational}
Let $X$ be a smooth compact manifold. Then $X$ is rationally connected with nef anticanonical bundle 
if and only if $T_{X}$ is generically $\omega^{n-1}$-strictly positive for any Kähler metric $\omega$.
\end{question}

As a second application, we study a Conjecture of Y.Kawamata 
(cf. \cite[Thm. 1.1]{Mi87} for the dual case and \cite{Xie05} for dimension $3$.)

\begin{conj}
If $X$ is a compact Kähler manifold with nef anticanonical bundle.
Then  
$$\int_{X} c_{2}(T_{X})\wedge\omega_{1}\wedge\cdots\wedge\omega_{n-2}\geq 0 $$
for all nef classes $\omega_{1},\cdots, \omega_{n-1}$.
\end{conj}

Using Theorem \ref{mainthmantinef}, we can prove

\begin{prop}\label{thmcalcul}
Let $(X,\omega_X )$ be a compact K\"{a}hler manifold with nef anticanonical bundle 
and let $\omega_X$ be a K\"{a}hler metric.
Then 
\begin{equation}\label{inequalitychern}
\int_X c_{2}(T_X)\wedge (c_1 (-K_X)+\epsilon\omega_X )^{n-2}\geq 0 
\end{equation}
for any $\epsilon > 0$ small enough.
\end{prop}

\begin{proof}

Let $\nd$ be the numerical dimension of $-K_X$. 
Let
\begin{equation}\label{filtrationthm}
0 = \sF_0 \subset \sF_1 \subset \ldots \subset \sF_l = T_X
\end{equation}
be a stable filtration of the semi-stable filtration of $T_X$ with respect to the polarization
$(c_1(-K_X) + \epsilon\omega_X)^{n-1}$ for some small $\epsilon> 0$.
By Proposition \ref{lemmastable}, 
the filtration \eqref{filtrationthm} is independent of $\epsilon$ when $\epsilon\rightarrow 0$.
By Theorem \ref{mainthmantinef}, we have
$$c_1 (\sF_i /\sF_{i-1})\wedge (-K_X)^{\nd}\wedge(\omega_X)^{n-1-\nd} \geq 0 \qquad\text{ for any }i.$$
Since 
$$\sum_{i}c_1 (\sF_i /\sF_{i-1})\wedge (-K_X)^{\nd}\wedge(\omega_X)^{n-1-\nd}=(-K_X)^{\nd +1}\wedge(\omega_X)^{n-1-\nd}=0 ,$$
we obtain
\begin{equation}\label{equationfirstderiv}
c_1 (\sF_i /\sF_{i-1})\wedge (-K_X)^{\nd}\wedge(\omega_X)^{n-1-\nd} = 0 \qquad\text{ for any }i.
\end{equation}
Combining \eqref{equationfirstderiv} with Theorem \ref{mainthmantinef}, we obtain
\begin{equation}\label{equationsecondderiv}
 c_1 (\sF_i /\sF_{i-1})\wedge (-K_X)^{\nd -1}\wedge(\omega_X)^{n-\nd} \geq 0 \qquad\text{ for any }i.
\end{equation}
Combining this with the stability condition of the filtration, we can find an integer $k\geq 1$ 
such that
\begin{equation}\label{equationsecondstrict}
c_1 (\sF_i /\sF_{i-1})\wedge (-K_X)^{\nd -1}\wedge(\omega_X)^{n-\nd} = a_i > 0 \qquad\text{ for  }i\leq k ,
\end{equation}
and 
$$ c_1 (\sF_i /\sF_{i-1})\wedge (-K_X)^{\nd -1}\wedge(\omega_X)^{n-\nd} = 0 \qquad\text{ for  }i> k .$$

We begin to prove \eqref{inequalitychern}.
Set $r_i := \rank (\sF_i /\sF_{i-1})$. 
By L\"{u}bke's inequality (cf. the proof of \cite[Thm. 6.1]{Mi87}), we have
\begin{equation}\label{Miyaokainequality}
c_2 (T_X)\wedge (-K_X +\epsilon\omega_X)^{n-2}
\end{equation}
$$
\geq (c_1 (-K_X) ^{2}-
\sum_{i} \frac{1}{r_i} c_1 (\sF_i /\sF_{i-1}) ^2)(-K_X +\epsilon\omega_X)^{n-2} .
$$
There are three cases.
\smallskip

{\em Case (1): $\sum\limits_{i\leq k} r_i \geq 2$ and $\nd \geq 2$ .}
Using the Hodge index theorem, we have
\footnote{It is important that $\alpha ^2\wedge (-K_X +\epsilon\omega_X)^{n-2}$ maybe negative.}
\begin{equation}\label{hodgeindex}
 ( \alpha ^2\wedge (-K_X +\epsilon\omega_X)^{n-2}) ((-K_X)^2 \wedge (-K_X +\epsilon\omega_X)^{n-2})
\end{equation}
$$\leq (\alpha\wedge (-K_X)\wedge (-K_X +\epsilon\omega_X)^{n-2})^{2} ,$$
for any $\alpha\in H^{1,1}(X ,\mathbb{R})$.
If we take $\alpha=c_1 (\sF_i /\sF_{i-1} )$ in \eqref{hodgeindex} and use \eqref{Miyaokainequality}, 
we obtain
\begin{equation}\label{estimatetwoterm}
c_2 (T_X)\wedge (-K_X +\epsilon\omega_X)^{n-2}\geq
\end{equation}
$$ c_1 (-K_X) ^{2}\wedge (-K_X +\epsilon\omega_X)^{n-2} 
-\sum_{i\leq k} \frac{1}{r_i}\frac{ (c_1 (\sF_i /\sF_{i-1})\wedge (-K_X)\wedge (-K_X +\epsilon\omega_X)^{n-2})^{2}  }{(-K_X) ^{2} \wedge (-K_X +\epsilon\omega_X)^{n-2}}
$$
Now we estimate the two terms in the right hand side of \eqref{estimatetwoterm}.
Using \eqref{equationsecondstrict}, we have
$$ c_1 (-K_X) ^{2}\wedge (-K_X +\epsilon\omega_X)^{n-2}= (\sum_{i\leq k} a_i )\epsilon^{n-\nd}+ O(\epsilon^{n-\nd}) $$
and 
$$\sum_{i\leq k} \frac{1}{r_i}\frac{ (c_1 (\sF_i /\sF_{i-1})\wedge (-K_X)\wedge (-K_X +\epsilon\omega_X)^{n-2})^{2}  }{(-K_X) ^{2} \wedge (-K_X +\epsilon\omega_X)^{n-2}}$$
$$=\frac{1}{\sum_{i\leq k} a_i} (\sum_{i\leq k} \frac{a_i ^2}{r_i})\cdot \epsilon^{n-\nd}+ O(\epsilon^{n-\nd}) .$$
Since $\sum\limits_{i\leq k} r_i \geq 2$, we have 
$$\sum_{i\leq k} a_i > \frac{1}{\sum_{i\leq k} a_i} (\sum_{i\leq k} \frac{a_i ^2}{r_i}).$$
Therefore $c_2 (T_X)\wedge (-K_X +\epsilon\omega_X)^{n-2}$ is strictly positive when $\epsilon> 0$ is small enough.
\smallskip

{\em Case (2): $\sum\limits_{i\leq k} r_i =1$ and $\nd \geq 2$ .}
In this case, we obtain immediately that $r_1=1 $ and $ k=1$. 
Moreover, \eqref{equationsecondstrict} in this case means that
$$c_1 (\sF_1 )\wedge (-K_X)^{\nd -1}\wedge(\omega_X)^{n-\nd}>0 ,$$
and
\begin{equation}\label{secondzero}
c_1 (\sF_i /\sF_{i-1})\wedge (-K_X)^{\nd -1}\wedge(\omega_X)^{n-\nd}=0 \qquad\text{ for } i\geq 2.
\end{equation}
Assume that $s$ is the smallest integer such that
$$c_1 (\sF_2 /\sF_1 )\wedge (-K_X)^{\nd -s}\wedge(\omega_X)^{n-\nd +s -1}>0 .$$
Taking $\alpha= c_1 (\sF_i /\sF_{i-1})$ in \eqref{hodgeindex} for any $i\geq 2$,
we get 
\begin{equation}\label{adddvv1}
c_1 (\sF_i /\sF_{i-1}) ^{2}\wedge (c_1(-K_X) + \epsilon\omega_X)^{n-2}
\end{equation}
$$\leq
\frac{(c_1 (\sF_i /\sF_{i-1})\wedge (-K_X)\wedge (-K_X +\epsilon\omega_X)^{n-2})^{2}}
{(-K_X)^2 \wedge (-K_X +\epsilon\omega_X)^{n-2}}
$$
$$\leq\frac{(\epsilon^{n+s-\nd -1})^2}{\epsilon^{n-\nd}}(1 + O (1) )= \epsilon^{2s+n-\nd -2}+ O (\epsilon^{2s+n-\nd -2}) 
\text{ for }i\geq 2 .$$
Similarly, if we take $\alpha=\sum\limits_{i\geq 2} c_1 (\sF_i /\sF_{i-1})$ in \eqref{hodgeindex},
we obtain
\begin{equation}\label{adddvv2}
(\sum_{i\geq 2}c_1 (\sF_i /\sF_{i-1}) )^{2}\wedge (c_1(-K_X) + \epsilon\omega_X)^{n-2}\leq \epsilon^{2s+n-\nd -2} .
\end{equation}
Combining \eqref{adddvv1}, \eqref{adddvv2} with \eqref{Miyaokainequality}, 
we obtain
$$c_2 (T_X)\wedge (-K_X +\epsilon\omega_X)^{n-2}$$
$$\geq 
(c_1 (-K_X) ^{2}-\sum_{i\geq 2} \frac{1}{r_i} c_1 (\sF_i /\sF_{i-1}) ^2 -(c_1 (-K_X) -\sum_{i\geq2} c_1 (\sF_i /\sF_{i-1}))^2)
(-K_X +\epsilon\omega_X)^{n-2}$$
$$=2c_1 (-K_X)\wedge (\sum_{i\geq 2} c_1 (\sF_i /\sF_{i-1})) \wedge (-K_X +\epsilon\omega_X)^{n-2}$$
$$
-(\sum_{i\geq 2} \frac{1}{r_i} c_1 (\sF_i /\sF_{i-1}) ^2 + (\sum_{i\geq2} c_1 (\sF_i /\sF_{i-1}))^{2})\wedge (-K_X +\epsilon\omega_X)^{n-2}$$
$$\geq \epsilon^{n-\nd+s-1} -\epsilon^{n-\nd +2s-2} .$$
Let us observe that by \eqref{secondzero} we have $s\geq 2$.
Therefore $c_2 (T_X)\wedge (-K_X +\epsilon\omega_X)^{n-2}$ is strictly positive for $\epsilon> 0$ small enough.
\smallskip

{\em Case (3): $\nd =1$. }
Using \eqref{Miyaokainequality}, we have 
$$c_2 (T_X)\wedge (-K_X +\epsilon\omega_X)^{n-2}\geq -
\sum_{i} \frac{1}{r_i} c_1 (\sF_i /\sF_{i-1}) ^2 (-K_X +\epsilon\omega_X)^{n-2} .$$
By the Hodge index theorem, we obtain 
$$c_2 (T_X)\wedge (-K_X +\epsilon\omega_X)^{n-2} 
$$
$$\geq \lim_{t\rightarrow 0 ^+} -\sum_{i} \frac{1}{r_i}\frac{ (c_1 (\sF_i /\sF_{i-1})\wedge (-K_X+ t\omega_X)
\wedge (-K_X +\epsilon\omega_X)^{n-2})^{2}  }{(-K_X+t\omega_X) ^{2} \wedge (-K_X +\epsilon\omega_X)^{n-2}}.$$
Let us observe that by \eqref{equationfirstderiv} we have 
$$c_1 (\sF_i /\sF_{i-1})\wedge (-K_X) \wedge(\omega_X)^{n-2} = 0 \qquad\text{ for any }i .$$ 
Then
$$c_2 (T_X)\wedge (-K_X +\epsilon\omega_X)^{n-2}$$
$$\geq \lim_{t\rightarrow 0 ^+} -\sum_{i} \frac{1}{r_i}\cdot
\frac{(t\epsilon^{n-2} c_1 (\sF_i /\sF_{i-1})\wedge\omega_X^{n-1}) ^{2}}{t^{2}\epsilon^{n-2}\omega_X^{n}+t\epsilon^{n-2} (-K_X)\omega_X^{n-1}}=0.$$

\end{proof}

It is interesting to study the case when the equality holds in \eqref{inequalitychern} of Proposition \ref{thmcalcul}.
We will prove that in this case, $X$ is either a torus or a smooth $\PP^1$-fibration over a torus.
Before proving this result, we first prove an auxiliary lemma.

\begin{lemme}\label{calucation}
Let $(X , \omega_X )$ be a compact K\"{a}hler manifold with nef anticanonical bundle.
Let 
\begin{equation}\label{filtrationlemma}
0 = \sF_0 \subset \sF_1 \subset \ldots \subset \sF_l = T_X
\end{equation}
be a stable subfiltration of the Jordan-H\"older filtration with respect to $(c_1 (-K_X) +\epsilon\omega_X )^{n-1}$.
If $\int_X c_2 (T_X)\wedge(c_1 (-K_X) +\epsilon\omega_X )^{n-2}=0$ for some $\epsilon> 0 $ small enough, 
we have 

{\em $(i)$} $\nd (-K_X)=1$.

{\em $(ii)$} $(\sF_i /\sF_{i-1})^{**}$ is projectively flat for all $i$, i.e., $(\sF_i /\sF_{i-1})^{**}$ is locally free and
there exists a smooth metric $h$ on it such that $i\Theta_h (\sF_i /\sF_{i-1})^{**} = \alpha\Id$, where $\alpha$ is a $(1,1)$-form.

{\em $(iii)$} $c_{2}(\sF_i /\sF_{i-1})=0$ for all $i$, 
and \eqref{filtrationlemma} is regular outside a subvariety of codimension at least $3$.

{\em $(iv)$}  $c_1 ( \sF_i /\sF_{i-1} ) =\rank (\sF_i /\sF_{i-1}) \cdot \alpha_{i}$
for some $\alpha_i \in H^{1,1}(X,\mathbb{Z})$. Moreover, $c_1 (\sF_i /\sF_{i-1})$
is nef and proportional to $c_1 (-K_X)$.

\end{lemme}

\begin{remark}\label{GRR}
We remind that for a vector bundle $V$ of rank $k$ supported on a subvariety $j: Z\subset X$ of codimension $r$,
by the Grothendieck-Riemann-Roch theorem, we have
$$c_{r}(j_* (V))=(-1)^{r-1} (r-1)! k [Z] .$$
Therefore for any torsion free sheaf $\sE$, we have $c_{2}(\sE)\geq c_{2}(\sE^{**})$ and the equality holds if and only if $\sE=\sE^{**}$
outside a subvariety of codimension at least $3$.
\end{remark}

\begin{proof}

By the proof of Proposition \ref{thmcalcul},  the equality
\begin{equation}\label{euqlitychernclass}
\int_X c_2 (T_X)\wedge(c_1 (-K_X) +\epsilon\omega_X )^{n-2}=0
\end{equation}
implies that the filtration \eqref{filtrationlemma} is in the case $(3)$, i.e., 
$\nd (-K_X)=1$.
By the proof of Proposition \ref{thmcalcul},
\eqref{euqlitychernclass} implies also that the filtration \eqref{filtrationlemma} 
satisfies the following three conditions:

\begin{equation}\label{equationdd1}
\int_X c_1(-K_X)^{2}\wedge(c_1 (-K_X) +\epsilon\omega_X )^{n-2}=0.
\end{equation}

\begin{equation}\label{equationdd2}
 \int_X c_{2}((\sF_i /\sF_{i-1})^{**})\wedge(c_1 (-K_X) +\epsilon\omega_X )^{n-2}=\int_X c_2 (\sF_i /\sF_{i-1})\wedge(c_1 (-K_X) +\epsilon\omega_X )^{n-2}=0.
\end{equation}

\begin{equation}\label{equation3}
\int_X c_1 (\sF_i /\sF_{i-1}) ^{2}\wedge(c_1 (-K_X) +\epsilon\omega_X )^{n-2}=0 .
\end{equation}

By \cite[Cor 3]{BS94}, \eqref{equationdd2} and \eqref{equation3} imply that $(\sF_i /\sF_{i-1})^{**}$ 
is locally free and projectively flat.
$(ii)$ is proved.
$(iii)$ follows by \eqref{equationdd2} and the Remark \ref{GRR}.
We now check $(iv)$. 
$(ii)$ implies that $c_1 ( \sF_i /\sF_{i-1} ) =\rank (\sF_i /\sF_{i-1}) \cdot \alpha_{i}$
for some $\alpha_i \in H^{1,1}(X,\mathbb{Z})$. 
By \eqref{equationfirstderiv}, we have
$$\int_X c_1 (-K_X)\wedge c_1 (\sF_i /\sF_{i-1}) \wedge (c_1 (-K_X) +\epsilon\omega_X )^{n-2}= 0 .$$
Combining this with \eqref{equationdd1} and \eqref{equation3}, by the Hodge index theorem
\footnote{In fact, let $Q (\alpha,\beta )=\int_X \alpha\wedge\beta \wedge (c_1 (-K_X) +\epsilon\omega_X )^{n-2}$.
Then $Q$ is of index $(1, m)$. Let $V$ be the subspace of $H^{1,1}(X, \mathbb{R})$ where $Q$ is negative definite.
If $Q (\alpha_1, \alpha_1)=Q (\alpha_2, \alpha_2)=Q (\alpha_1 ,\alpha_2)=0$ for some non trivial $\alpha_1 ,\alpha_2$,
then both $\alpha_1$ and $\alpha_2$ are not contained in $V$.
Therefore we can find a $t\in \mathbb{R}$, such that $(\alpha_1 -t\alpha_2)\in V$.
Since $Q (\alpha_1 -t\alpha_2, \alpha_1 -t\alpha_2)=0$, we get $\alpha_1 -t\alpha_2=0$.
Therefore $\alpha_1$ is proportional to $\alpha_2$. },
we obtain that
$c_1 (\sF_i /\sF_{i-1})=a_i\cdot c_1 (-K_X)$ for some $a_i\in \mathbb{Q}$.
By Theorem \ref{mainthmantinef}, 
we have
$$c_1 (\sF_i /\sF_{i-1}) \wedge (c_1(-K_X)+\epsilon\omega_X)^{n-1}\geq 0 .$$
Therefore $a_i \geq 0$ and $\sF_i /\sF_{i-1}$ is nef.
\end{proof}

Using an idea of A. H\"{o}ring, we finally prove that

\begin{prop}
Let $(X , \omega_X )$ be a projective manifold with nef anticanonical bundle. 
We suppose that $\int_X c_2 (X)\wedge(c_1 (-K_X) +\epsilon\omega_X )^{n-2}=0$ for some $\epsilon> 0 $ small enough.
Then after a finite \'{e}tale cover,
$X$ is either a torus or a smooth $\PP^1$-fibration over a torus.
\end{prop}

\begin{proof}

Denote by $m \in \N$ the index of $-K_X$, that is $m$ is the largest positive integer such that there exists a Cartier divisor $L$ with $mL \equiv -K_X$. 
Let 
\begin{equation}\label{filtration}
0 = \sF_0 \subset \sF_1 \subset \ldots \subset \sF_l = T_X
\end{equation}
be the Jordan-H\"older filtration with respect to $\omega_X^{n-1}$.
Let $Z$ be the locus where \eqref{filtration} is not regular.
By Lemma \ref{calucation}, we have $\codim Z\geq 3$. 
\smallskip

{\em 1st case. $m \geq 2$}. Since $K_X$ is not nef, there exists a Mori contraction $\holom{\varphi}{X}{Y}$. Since $nd(-K_X)=1$ and $-K_X$ is ample on all the $\varphi$-fibres,
we see that all the $\varphi$-fibres have dimension at most one. By Ando's theorem \cite{And85} we know that $\varphi$ is either a blow-up along a smooth subvariety of codimension two
or a conic bundle. Since $m \geq 2$ we see that the contraction has length at least two, so $\varphi$ is a conic bundle without singular fibres, i.e. a $\PP^1$-bundle.
By \cite[4.11]{Miy83} we have
$$
\varphi_* (K_X^2) = - 4 K_Y,
$$
so $K_X^2=0$ implies that $K_Y \equiv 0$. By the condition 
$$\int_X c_2 (X)\wedge(c_1 (-K_X) +\epsilon\omega_X )^{n-2}=0 ,$$
we obtain that $c_2 (Y)=0$. 
Therefore $Y$ is a torus and the proposition is proved.
\smallskip

{\em 2nd case. $m = 1$}

By $(iv)$ of Lemma \ref{calucation}, the condition $m=1$ implies that $\rank \sF_1 = 1$ and $\mu (\sF_i/ \sF_{i-1})=0$ for $i> 1$.
By the proof of Proposition \ref{propequivalent}, we get
\begin{equation}\label{equality1}
c_1 (\sF_i/ \sF_{i-1})=0 \qquad\text{ for } i> 1.
\end{equation}
We consider a Mori contraction: 
$$\varphi: X\rightarrow Y .$$ 
By \cite{And85},
there are two cases:
\smallskip

{\em Case (1):} $\varphi$ is a blow-up along a smooth subvariety $S\subset Y$ of codimension two. 

Let $E$ be the exceptional divisor. 
Since \eqref{filtration} is free outside $Z$ of codimension $\geq 3$, 
for a general fiber over $s\in S$, 
\eqref{filtration} is regular on the fiber $X_s$ over $s$, which is $\PP^1$.
By \eqref{equality1}, we know that
$$T_{X}|_{\PP^1}=\mathcal{O}_{\PP^1}(a)\oplus \mathcal{O}_{\PP^1} ^{n-1},$$
for some $a> 0$.
On the other hand, over $\PP^1$, we have a direct decomposition
$$T_X|_{\PP^1}= T_E |_{\PP^1} \oplus N_{X/E} |_{\PP^1}=T_E |_{\PP^1}\oplus [E] |_{\PP^1}.$$
Since $[-E]|_{\PP^1}$ is strictly positive,
$T_X|_{\PP^1}$ must contain a strictly negative part. 
We get a contradiction.
\smallskip

{\em Case (2):} $\varphi$ is a conic bundle, and $Y$ is smooth.

We consider the reflexive subsheaf $(T_{X/Y})^{**}$ of $T_X$.
We first prove that 
\begin{equation}\label{relativetangent}
(T_{X/Y})^{**} = \sF_1 .
\end{equation}
By \eqref{equality1}, we have
\begin{equation}\label{equality2}
c_1 (\sF_1)=c_1 (-K_X) .
\end{equation}
Let $y$ be a generic point in $Y\setminus \pi(Z)$ 
(i.e., $\varphi_y=\PP^1$ and \eqref{filtration} is regular on $\varphi_y$). 
Since \eqref{filtration} is free over $\varphi_y$, by \eqref{equality1} and \eqref{equality2}
we obtain that $\sF_1= ( T_{X/Y} )^{**}$ over $\varphi_y$.
Since both $\sF_1$ and $( T_{X/Y} )^{**}$ are immersed as vector subbundles in 
$T_X$ outside a subvariety of codimension at least $3$,
combining with \eqref{relativetangent},
we obtain that $\sF_1= (T_{X/Y})^{**}$ outside this subvariety. 
Then the reflexiveness of $\sF_1$ and $(T_{X/Y})^{**}$ 
implies that $\sF_1= (T_{X/Y})^{**}$ on $X$.

We now prove that $T_X/ \sF_1 = \varphi^{*} (T_Y)$ outside a subvariety of codimension at least $3$.
Let $\widetilde{Z}\subset Y$ be the locus where the fiber is non reduced.
By \cite[Thm 3.1]{And85}, for any $y\in \widetilde{Z}$, we have $\varphi_y=2 C$, 
where $C\simeq\PP^1$ and $N_{C/X}$ is not trivial. 
Then $C\cap Z\neq \emptyset$. 
Recall that $Z$ is the singular set of the filtration \eqref{filtration} of codimension at least $3$.
Therefore the codimension $ \widetilde{Z}$ in $Y$ is at least $2$.
Therefore $T_X/ \sF_1 = \varphi^{*} (T_Y)$ outside a subvariety of codimension at least $2$.
Since $T_X/ \sF_1$ and $\varphi^{*} (T_Y)$ are locally free outside a subvariety of codimension at least $3$ (thus reflexive on this open set),
we obtain that $T_X/ \sF_1 = \varphi^{*} (T_Y)$ outside a subvariety of codimension at least $3$.

As consequence, we have 
$$\varphi^{*}(c_1(-K_Y))=c_1 (T_X/ \sF_1) =0 ,$$
where the last equality comes from \eqref{equality1}.
Since $\varphi$ is surjective and $X, Y$ are compact K\"{a}hler, we get $c_1 (-K_Y)=0$.
By Beauville's decomposition, after a finite \'{e}tale cover, we can suppose that $Y$ is a direct product
$T\times Y_1$, where $T$ is a torus and $Y_1$ is a product of Calabi-Yau and hyperk\"ahler manifolds.
If $Y_1$ is non trivial, we have $c_2 (Y)> 0$. 
But $c_{2}(T_X/ \sF_1)=c_{2}(T_X/ \sF_1)=c_2 (T_Y)$ by the above argument,  
we get $c_2 (X)> 0$. We get a contradiction.
Therefore $Y$ is a torus. 
By \cite[Thm 1.3]{CH13}, $\varphi$ admits a smooth fibration to $Y$ and the fibers are $\PP^1$.  

\end{proof}

\begin{remark}
In general, if $\int_X c_2 (X)\wedge(c_1 (-K_X) +\epsilon\omega_X )^{n-2}=0$,
We cannot hope that $X$ can be covered by a torus.
In fact, the example \cite[Example 3.3]{DPS94} satisfies the equality $c_2 (X)=0$ and $X$ can not be decomposed as direct product 
of torus with $\PP^1$. Using \cite{DHP08}, we know that $X$ cannot be covered by torus.
Therefore we propose the following conjecture, which is a mild modification of the question of S.-T. Yau:
\end{remark}

\begin{conj}
Let $( X ,\omega_X )$ be a compact K\"{a}hler manifold with nef anticanonical bundle.
Then $\int_X c_2(T_X)\wedge\omega_X ^{n-2}\geq 0$. 
If the equality holds for some K\"{a}hler metric,
then $X$ is either a torus or a smooth $\PP^1$-fibration over a torus. 
\end{conj}

\begin{remark}
If one could prove that $T_X$ is generically nef with respect to the polarization 
$( c_1 (-K_X), \omega,\cdots , \omega )$,
using the same argument as in this section, one could prove this conjecture.
\end{remark}

\section{Surjectivity of the Albanese map}

As an application of Theorem \ref{mainthmantinef}, 
we give a new proof of the surjectivity of Albanese map 
when $X$ is a compact Kähler manifold with nef anticanonical bundle.

\begin{prop}\label{surjectivealb}
Let $(X, \omega)$ be a compact Kähler manifold with nef anticanonical bundle. 
Then the Albanese map is surjective, and smooth outside a subvariety of codimension at least 2. 
In particular, the fibers of the Albanese map are connected and reduced in codimension 1.
\end{prop}

\begin{proof}
Let 
\begin{equation}\label{filtrationalbanesemap}
0\subset \mathcal{E}_{0}\subset \mathcal{E}_{1}\subset\cdots\subset \mathcal{E}_{s}=T_{X}
\end{equation}
be a filtration of torsion-free subsheaves such that $\mathcal{E}_{i+1}/\mathcal{E}_{i}$
is an $\omega$-stable torsion-free subsheaf of $T_{X}/ \mathcal{E}_{i}$ of maximal slope.

{\em Case 1: \eqref{filtrationalbanesemap} is regular, i.e., all $\mathcal{E}_{i}$ and $\mathcal{E}_{i}/\mathcal{E}_{i-1}$ are locally free} 

In this case, we can prove that the Albanese map is submersive.
Let $\tau\in H^{0}(X, T_{X}^{*})$ be a nontrivial element.
To prove that the Albanese map is submersive, it is sufficient to prove that
$\tau$ is non vanishing everywhere.
Thanks to Theorem \ref{mainthmantinef} 
and the stability condition of $\mathcal{E}_{i}/\mathcal{E}_{i-1}$,
we can find a smooth metric $h_{i}$ on $\mathcal{E}_{i}/\mathcal{E}_{i-1}$
such that 
$$\frac{i \Theta_{h_{i}}(\mathcal{E}_{i}/\mathcal{E}_{i-1})\wedge \omega^{n-1}}{\omega^{n}}
=\lambda_{i}\cdot\Id_{\mathcal{E}_{i}/\mathcal{E}_{i-1}}$$
for some constant $\lambda_{i}\geq 0.$
Thanks to the construction of $\{h_{i}\}$ and Lemma \ref{lemmaestimatexten},
for any $\epsilon> 0$, 
there exists a smooth metric $h_{\epsilon}$ on $T_{X}$, 
such that
\begin{equation}\label{equation66}
\frac{i\Theta_{h_{\epsilon}}(T_{X})\wedge \omega^{n-1}}{\omega^{n}}\geq -\epsilon\cdot\Id_{T_{X}},
\end{equation}
and the matrix valued $(1,1)$-form $i\Theta_{h_{\epsilon}}(T_{X})$ is uniformly bounded.
Let $h_{\epsilon}^{*}$ be the dual metric on $T_{X}^{*}$.
Then the closed $(1,1)$-current 
$$T_{\epsilon}=\frac{i}{2\pi}\partial\overline{\partial}\ln \| \tau\|_{h_{\epsilon}^{*}} ^{2}$$
satisfies
\begin{equation}\label{equation67}
T_{\epsilon}\geq -\frac{\langle i\Theta_{h_{\epsilon}^{*}}(T_{X}^{*})\tau, \tau\rangle_{h_{\epsilon}^{*}}}
{\|\tau\|_{h_{\epsilon}^{*}} ^{2}}.
\end{equation}
Since $-\Theta_{h_{\epsilon}^{*}}(T_{X}^{*})=\leftidx{^t}{\Theta_{h_{\epsilon}}(T_{X})} ,$
\eqref{equation66} and \eqref{equation67} imply a pointwise estimate
\begin{equation}\label{equation68}
T_{\epsilon}\wedge\omega^{n-1}\geq -\epsilon\omega^{n}.
\end{equation}

We suppose by contradiction that $\tau (x)=0$ for some point $x\in X$, 
By Lemma \ref{lemmaestimatexten}, 
$i\Theta_{h_{\epsilon}}(T_{X})$ is uniformly lower bounded.
Therefore, there exists a constant $C$ such that
$T_{\epsilon}+C\omega$ is a positive current for any $\epsilon$.
After replacing by a subsequence, 
we can thus suppose that $T_{\epsilon}$ converge weakly to a current $T$,
and $T+C\omega$ is a positive current.
Since $\tau (x)=0$, we have  
$$\nu (T_{\epsilon}+C\omega, x)\geq 1 \qquad\text{for any }\epsilon,$$
where $\nu (T_{\epsilon}+C\omega, x)$ is the Lelong number of the current $T_{\epsilon}+C\omega$ at $x$.
Using the main theorem in \cite{Siu74},
we obtain that $\nu (T+C\omega, x)\geq 1$.
Therefore there exists a constant $C_{1}> 0$ such that
$$\int_{B_{x}(r)} (T+C\omega)\wedge \omega^{n-1}\geq C_{1}\cdot r^{2n-2}\qquad\text{for }r\text{ small enough,}$$
where $B_{x}(r)$ is the ball of radius $r$ centered at $x$.
Then 
$$\int_{U_{x}} T\wedge \omega^{n-1}> 0$$
for some neighborhood $U_{x}$ of $x$.
Therefore
\begin{equation}\label{equation69}
\lim_{\epsilon\rightarrow 0}\int_{U_{x}} T_{\epsilon}\wedge \omega^{n-1}> 0.
\end{equation}
Combining \eqref{equation68} with \eqref{equation69}, we obtain
$$\lim_{\epsilon\rightarrow 0}\int_{X} T_{\epsilon}\wedge\omega^{n-1}> 0 .$$
We get a contradiction by observing that all $T_{\epsilon}$ are exact forms.

{\em Case 2: General case}

By Lemma \ref{lemmaregular},
there exists a desingularization $\pi: \widetilde{X}\rightarrow X$, 
such that 
$\pi^{*}(T_{X})$ admits a filtration:
$$0\subset E_{1}\subset E_{2}\subset\cdots \subset \pi^{*}(T_{X})$$
satisfying that $E_{i}, E_{i}/ E_{i-1}$ are vector bundles 
and $\pi_{*}(E_{i})=\mathcal{E}_{i}$ on $X\setminus Z$, 
where $Z$ is an analytic subset of codimension at least 2.
Let $\tau\in H^{0}(X, T_{X}^{*})$ be a nontrivial element.
Our aim is to prove that $\tau$ is non vanishing outside $Z$.

Let $x\in \widetilde{X}\setminus \pi^{-1}(Z)$. 
Let $U_x$ be a small neighborhood of $x$ such that $U_x \subset \widetilde{X}\setminus \pi^{-1}(Z)$.
We suppose by contradiction that $\pi^{*}(\tau) (x)=0 $.
By \cite{BS94}, there exists Hermitian-Einstein metrics $h_{\epsilon, i}$ on $E_{i}/E_{i-1}$
with respect to $\pi^{*}\omega+\epsilon\omega_{\widetilde{X}}$,
and $\{ i\Theta_{h_{\epsilon, i}}(E_{i}/E_{i-1})\}_{\epsilon}$ 
is uniformly bounded on $U_x$
\footnote{In fact, \cite{BS94} proved that $h_{\epsilon, i}$ 
and $h^{-1}_{\epsilon, i}$ are $C^{1,\alpha}$-uniform bounded in $U_x$.
Since $U_x$ is in $X\setminus Z$, 
$\omega_{\epsilon}:=\pi^{*}\omega+\epsilon\omega_{\widetilde{X}}$ is uniformly strict positive on $U_{x}$.
By \cite[Chapter I, (14.16)]{Kob87} and Hermitian-Einstein condition, we obtain that
$\Delta_{\omega_{\epsilon}} ( h_{\epsilon,i})_{j,k}$ is uniformly $C^{\alpha}$ bounded on $U_x$,
where $\Delta_{\omega_{\epsilon}}$ is the Laplacian with respect to $\omega_{\epsilon}$ and $(h_{\epsilon,i})_{j,k} :=h_{\epsilon,i}(e_j ,e_k)$
for a fixed base $\{e_k\}$ of $E_i/E_{i-1}$.
The standard elliptic estimates gives the uniform boundedness of $i\Theta_{h_{\epsilon, i}}(E_i/E_{i-1})$ on $U_x$.}.
Combining this with Lemma \ref{lemmaestimatexten}, we can construct a smooth metric $h_{\epsilon}$ on $\pi^{*}(T_{X})$ such that
\begin{equation}\label{equation70}
\frac{i \Theta_{h_{\epsilon}}(\pi^{*}(T_{X}))\wedge (\pi^{*}\omega+\epsilon\omega_{\widetilde{X}})^{n-1}}
{(\pi^{*}\omega+\epsilon\omega_{\widetilde{X}})^{n}}
\geq -2 C\epsilon\cdot\Id_{\pi^{*}(T_{X})} , 
\end{equation}
and $i \Theta_{h_{\epsilon}}(\pi^{*}(T_{X}))$ is uniformly bounded on $U_{x}$.
Let $T_{\epsilon}=\frac{i}{2\pi}\partial\overline{\partial}\ln \|\pi^{*}(\tau)\|^{2}_{h_{\epsilon}^{*}}$.
By the same argument as in Case 1, the uniform boundedness of $i\Theta_{h_{\epsilon}}(\pi^* (T_X))$ in a neighborhood of $x$
implies the existence of a neighborhood $U'_{x}$ of $x$ and a constant $c> 0$, such that
$$\lim_{\epsilon\rightarrow 0}
\int_{U'_{x}}T_{\epsilon}\wedge(\pi^*(\omega )+\epsilon\omega_{\widetilde{X}})^{n-1}\geq c .$$
Combining this with \eqref{equation70}, we get
$$\lim_{\epsilon\rightarrow 0}
\int_{\widetilde{X}}T_{\epsilon}\wedge(\pi^*(\omega)+\epsilon\omega_{\widetilde{X}})^{n-1}\geq c ,$$
which contradicts with the fact that all $T_{\epsilon}$ are exact.
Therefore $\tau$ is non vanishing outside $Z$.
Proposition \ref{surjectivealb} is proved.
\end{proof}

\section{Appendix}\label{bochner}

We would like to give a proof of the implication $(iii)\Rightarrow (ii)$ in Proposition \ref{propequivalent} 
without using \cite[Theorem 0.1]{BM01}.

\begin{proof}
By \cite[Criterion 1.1]{CDP12}, 
to prove the implication, 
it is sufficient to prove that
for some ample line bundle $F$ on $X$, there exists a constant $C_{F}> 0$, 
such that
\begin{equation}\label{equation111}
H^{0}(X, (T_{X}^{*})^{\otimes m}\otimes F^{\otimes k})=0 
\qquad\text{for all }m, k \text{ with }m\geq C_{F}\cdot k.
\end{equation}

Thanks to the condition $(iii)$, 
there exists a K\"{a}hler class $A$,
such that 
$$\mu_{A}(\mathcal{F}_{i}/\mathcal{F}_{i-1})\geq c \qquad\text{for all } i, $$
for some constant $c>0$.
Moreover, for the Harder-Narasimhan filtration of $(T_{X})^{\otimes m}$ with respect to $A$,
$m\cdot c$ is also a lower bound of the minimal slope with respect to the filtration.

We now prove \eqref{equation111} by a basic Bochner technique.
After replacing by a more refined filtration, we can suppose that 
\begin{equation}\label{equation112}
0\subset \mathcal{E}_{0}\subset \mathcal{E}_{1}\subset\cdots\subset \mathcal{E}_{s}=(T_{X})^{\otimes m}
\end{equation}
is a filtration of torsion-free subsheaves such that $\mathcal{E}_{i+1}/\mathcal{E}_{i}$
is an $\omega$-stable torsion-free subsheaf of $T_{X}/\mathcal{E}_{i}$ of maximal slope
for simplicity.
Let $\omega$ be a positive $(1,1)$-form representing $c_{1}(A)$.

If all the quotients of the filtration \eqref{equation112} are free, 
then there exists a Hermitian-Einstein metric on every quotient. 
Since $\mu_{A}(\mathcal{E}_{i}/\mathcal{E}_{i-1})\geq c\cdot m$,
thanks to Lemma \ref{lemmaestimatexten}, 
we can construct a smooth metric $h$ on $(T_{X})^{\otimes m}$, 
such that
\begin{equation}\label{equation113}
\frac{ i\Theta_{h} (T_{X}^{\otimes m})\wedge\omega^{n-1}}{\omega^{n}}\geq \frac{m\cdot c}{2} \Id .
\end{equation}
Let $\tau\in H^{0}(X, (T_{X}^{*})^{\otimes m}\otimes F^{\otimes k})$.
We have
\begin{equation}\label{equation114}
\Delta_{\omega} (\|\tau\|^{2}_{h^{*}})=
\|D_{h}' \tau\|^{2}-\frac{\langle  i\Theta_{h^{*}} ((T_{X}^{*})^{\otimes m}\otimes F^{\otimes k} )\tau, 
\tau\rangle \wedge \omega^{n-1}}{\omega^{n}}. 
\end{equation}
If $m\geq C_{F}\cdot k$ for some constant $C_{F}$ big enough with respect to $c$, 
\eqref{equation113} implies that 
$$\int_{X}\|D_{h}' \tau\|^{2}\omega^{n}-\langle  i\Theta_{h^{*}} ((T_{X}^{*})^{\otimes m}\otimes F^{\otimes k} )\tau, 
\tau\rangle \wedge \omega^{n-1}\geq c_{1}\|\tau\|^{2}_{h^{*}}$$
for some constant $c_{1} > 0$.
Observing moreover that 
$$\int_{X}\Delta_{\omega} (\|\tau\|^{2}_{h^{*}}) \omega^{n}=0 ,$$
then $\tau=0$.

If the quotients $\mathcal{E}_{i}/\mathcal{E}_{i-1}$ of \eqref{equation112} 
are not necessary free,
by Lemma \ref{lemmaregular}, 
we can find a resolution
$\pi: \widetilde{X}\rightarrow X$ such that
there exists a filtration 
$$0\subset E_{1}\subset E_{2}\subset\cdots \subset \pi^{*}(T_{X})$$
where $E_{i}, E_{i}/E_{i-1}$ are vector bundles and 
$$\mu_{\pi^{*}(A)}(E_{i}/E_{i-1})=\mu_{A} (\mathcal{E}_{i}/\mathcal{E}_{i-1})\geq c\cdot m .$$
Thanks to the strict positivity of $c$, 
for $\epsilon$ small enough,
\begin{equation}\label{equation115}
\mu_{\epsilon} (E_{i}/E_{i-1})\geq \frac{c\cdot m}{2}\qquad\text{for any }i,
\end{equation}
where $\mu_{\epsilon}$ is the slope with respect to $\pi^{*}(A)+\epsilon \omega_{\widetilde{X}}$.
Thanks to the remark of Theorem \ref{mainthmantinef}, 
$E_{i}/E_{i-1}$ are also stable for $\pi^{*}(A)+\epsilon \omega_{\widetilde{X}}$ 
when $\epsilon$ small enough.
Therefore there exists a smooth Hermitian-Einstein metric on every quotient $E_{i}/E_{i-1}$.
Using Lemma \ref{lemmaestimatexten}, \eqref{equation115} implies that 
we can thus construct a smooth metric $h_{\epsilon}$ on $\pi^{*}(T_{X})^{\otimes m}$, 
such that
\begin{equation}\label{equation116}
\frac{i\Theta_{h_{\epsilon}} (\pi^{*} T_{X}^{\otimes m})\wedge (\pi^{*}(\omega)+\epsilon \omega_{\widetilde{X}})^{n-1} }
{(\pi^{*}(\omega)+\epsilon \omega_{\widetilde{X}})^{n}}
\geq \frac{m\cdot c}{4} \Id 
\end{equation}
for $\epsilon$ small enough.
Using the same Bochner technique on $\pi^{*}(T_{X})$ 
with respect to $\pi^{*}(A)+\epsilon \omega_{\widetilde{X}}$
as in \eqref{equation113} and \eqref{equation114}, 
we get
$$H^{0}(\widetilde{X}, \pi^{*}((T_{X}^{*})^{\otimes m}\otimes F^{\otimes k} ))=0 
\qquad\text{for all }m, k \text{ with }m\geq C_{F}\cdot k.$$
\eqref{equation111} is thus proved.
\end{proof}

\bibliographystyle{alpha}
\bibliography{biblio}

\end{document}